\documentclass[12pt]{amsart}

\usepackage{amssymb}
\usepackage{amscd}
\usepackage{statex}

\title{Galois lines for the Artin--Schreier--Mumford curve}
\author{Satoru Fukasawa}

\thanks{The author was partially supported by JSPS KAKENHI Grant Number JP19K03438}
\address{Department of Mathematical Sciences, Faculty of Science, Yamagata University, Kojirakawa-machi 1-4-12, Yamagata 990-8560, Japan}
\email{s.fukasawa@sci.kj.yamagata-u.ac.jp}

\newtheorem{thm}{Theorem}

\newtheorem{lem}{Lemma}

\newtheorem{fact}{Fact}

\theoremstyle{definition}
\newtheorem{rem}{Remark}

\begin{document}
\begin{abstract}
The arrangement of all Galois lines for the Artin--Schreier--Mumford curve in the projective $3$-space is described.  
It may be surprising that there exist infinitely many Galois lines intersecting this curve. 
\end{abstract}
\maketitle

\section{Introduction} 

In 1996, Hisao Yoshihara introduced the notion of a Galois point: 
for a plane curve $C \subset \mathbb{P}^2$ over an algebraically closed field $k$, a point $P \in \mathbb{P}^2$ is called a Galois point if the function field extension $k(C)/\pi_P^*k(\mathbb{P}^1)$ induced by the projection $\pi_P$ from $P$ is Galois (\cite{miura-yoshihara, yoshihara}). 
Analogously, Yoshihara introduced the notion of a {\it Galois line}: 
a line $\ell \subset \mathbb{P}^3$ is said to be Galois for a space curve $X \subset \mathbb{P}^3$, if the extension $k(X)/\pi_{\ell}^*k(\mathbb{P}^1)$ induced by the projection $\pi_{\ell}: X \dashrightarrow \mathbb{P}^1$ from $\ell$ is Galois (\cite{duyaguit-yoshihara, yoshihara2}). 
The following problem is raised in \cite{yoshihara-fukasawa}. 
\begin{itemize}
\item[(a)] Find Galois lines $\ell$ for $X$ in two cases where $X \cap \ell=\emptyset$ and $X \cap \ell \ne \emptyset$. 
\end{itemize}
It is hard to solve this problem, even for well studied curves. 
The present author and Higashine determined the arrangement of all Galois lines for the Giulietti--Korchm\'{a}ros curve (\cite{fukasawa-higashine}). 
However, this is only one known example of a curve in $\mathbb{P}^3$ of degree $d >3$ such that {\it all} Galois lines are determined (under the assumption that the automorphism group is non-trivial). 
Therefore, it would be good to provide such a curve. 

In this article, we consider the plane curve $C \subset \mathbb{P}^2$ defined by 
$$ (x^q-x)(y^q-y)=c, $$
where $p>0$ is the characteristic of $k$, $q \ge 3$ is a power of $p$, and $c \in k \setminus \{0\}$, 
which is called the Artin--Schreier--Mumford curve. 
This curve is important in the study of the automorphism groups of algebraic curves, since this is an ordinary curve and its automorphism group is large compared to its genus (see \cite{subrao}). 
Recently, the full automorphism group of this curve over any algebraically closed field $k$ was determined by Korchm\'{a}ros and Montanucci \cite{korchmaros-montanucci}. 
(For $q=p$, this result was obtained in \cite{valentini-madan}.)  
The smooth model of $C$ is denoted by $X$. 
It is known that the morphism 
$$ \varphi: X \rightarrow \mathbb{P}^3; \ (x: y : 1: xy)$$
is an embedding (see \cite[p.181]{subrao}, \cite{korchmaros-montanucci}). 
In this article, the arrangement of all Galois lines for $\varphi(X)$ is determined, as follows. 

\begin{thm} \label{main} 
Let $\ell \subset \mathbb{P}^3$ be a line, and let $G_{\ell}$ be the Galois group if $\ell$ is a Galois line. 
Then, $\ell$ is a Galois line for $\varphi(X)$ if and only if $\ell$ is one of the following: 
\begin{itemize}
\item[(a)] an $\mathbb{F}_q$-line in the plane $\{Z=0\}$, or
\item[(b)] the line defined by $W-a X=Y-a Z=0$ or $W- a Y=X-a Z=0$, where $a \in k$. 
\end{itemize}
For the case (a), there exist $q-1$ lines of them with $G_{\ell} \cong \mathbb{F}_q \rtimes C_2$, $q^2$ lines with $G_{\ell} \cong \mathbb{F}_q^* \rtimes C_2$, and two lines with $G_{\ell} \cong \mathbb{F}_q$, where $C_2$ is a cyclic group of order $2$. 
For the case (b), $G_{\ell}=G_{L_1}$ or $G_{L_2}$ in the automorphism group ${\rm Aut}(X)$ of $X$, where $L_1$ and $L_2$ are two Galois lines with $G_{L_1} \cong \mathbb{F}_q$ and $G_{L_2} \cong \mathbb{F}_q$ in (a). 
\end{thm} 

It is remarkable that Galois lines intersecting a space curve $X \subset \mathbb{P}^3$ can {\it not} be determined by subgroups of ${\rm Aut}(X)$, according to lines of type (b) in Theorem \ref{main} (see Remark \ref{infinitely many Galois lines} for further examples).  
An example of a space curve of genus $4$ and degree $6$ with infinitely many Galois lines was already presented by Komeda and Takahashi \cite{takahashi}.

\section{Preliminaries}
Since we investigate Galois lines, the following fact on Galois extensions is needed later (see \cite[III.7.1, III.7.2, III.8.2]{stichtenoth}). 

\begin{fact} \label{Galois extension}
Let $\pi: C \rightarrow C'$ be a surjective morphism of smooth projective curves. 
Assume that the field extension $k(C)/\pi^*k(C')$ is Galois. 
Then, the following hold. 
\begin{itemize} 
\item[(a)] The Galois group acts on each fiber of $\pi$ transitively. 
\item[(b)] For points $Q_1$ and $Q_2 \in C$ with $\pi(Q_1)=\pi(Q_2)$, the ramification indices are the same. 
\item[(c)] For each point $Q \in C$, the order of the stabilizer subgroup of $Q$ is equal to the ramification index at $Q$.  
\end{itemize}
\end{fact} 

The system of homogeneous coordinates on $\mathbb{P}^3$ is denoted by $(X:Y:Z:W)$. 
For $\mathbb{P}^2$, it is denoted by $(X: Y: Z)$. 
For different points $P, Q \in \mathbb{P}^3$, the line passing through $P$ and $Q$ is denoted by $\overline{PQ}$. 

Hereafter, we consider the Artin--Schreier--Mumford curve, which is denoted by $C$, except for Remark \ref{infinitely many Galois lines}. 
The curve $C$ is the image of the composite map of $\varphi$ and the projection from $(0:0:0:1) \in \mathbb{P}^3$.   
The set of all poles of $x$ (resp. of $y$) is denoted by $\Omega_1$ (resp. by $\Omega_2$), which coincides with the set of all zeros of $y^q-y$ (resp. $x^q-x$). 
The sets $\Omega_1$ and $\Omega_2$ consist of $q$ points.  
The pole of $x$ (resp. of $y$) corresponding to $y=\alpha \in \mathbb{F}_q$ (resp. $x=\alpha \in \mathbb{F}_q$) is denoted by $P_{\alpha}$ (resp. by $Q_\alpha$). 
For the point $P_{\alpha}$, $t=\frac{1}{x}$ is a local parameter at $P_{\alpha}$. 
It follows that
$$ \varphi=(x:y:1:x y)=(1:y/x:1/x: y)=(1:y t: t: y), $$
and $ \varphi(P_{\alpha})=(1:0:0:\alpha)$. 
Therefore, $q$ points of $\varphi(\Omega_1)$ (resp. of $\varphi(\Omega_2)$) are collinear, and the line $L_1$ (resp. $L_2$) containing $\varphi(\Omega_1)$ (resp. $\varphi(\Omega_2)$) is defined by $Y=Z=0$ (resp. $X=Z=0$). 
Note that 
$$ f_x \frac{dx}{dt}+f_y \frac{d y}{dt}=-(y^q-y)\frac{d x}{d t}-(x^q-x)\frac{d y}{d t}=0 $$
in $k(C)$, where $f(x, y)=(x^q-x)(y^q-y)-c$. 
This implies that 
$$ \frac{d y}{d t}(P_{\alpha})=0. $$
The tangent line at $\varphi(P_{\alpha})$ is spanned by points $\varphi(P_{\alpha})$ and 
\begin{eqnarray*} 
\left(0: \frac{d y}{d t}(P_{\alpha}) t(P_{\alpha})+y(P_{\alpha}):1:\frac{d y}{d t}(P_{\alpha})\right) &=&(0:\alpha:1:0).  
\end{eqnarray*}
Therefore, the tangent line at $\varphi(P_{\alpha})$ is defined by 
$$ W-\alpha X=Y-\alpha Z=0. $$

The following fact is needed for the proof of the only-if-part of our main theorem (see \cite[Section 4]{korchmaros-montanucci}). 

\begin{fact} \label{orders} 
Let $P \in \Omega_1 \cup \Omega_2$ and let $H \ni \varphi(P)$ be a hyperplane. 
Then, ${\rm ord}_P\varphi^*H=1, q$ or $q+1$.  
\end{fact} 

The following theorem on the full automorphism group ${\rm Aut}(X)$ was proved by Korchm\'{a}ros and Montanucci \cite{korchmaros-montanucci}.  

\begin{fact} \label{automorphism} 
For ${\rm Aut}(X)$, the following holds. 
\begin{itemize} 
\item[(a)] ${\rm Aut}(X)$ acts on $\Omega_1 \cup \Omega_2$ faithfully. 
\item[(b)] There exists an injective homomorphism ${\rm Aut}(X) \hookrightarrow PGL(4, k)$.
\item[(c)] $|{\rm Aut}(X)|=2q^2(q-1)$. 
\end{itemize}
\end{fact}

For the plane model $C \subset \mathbb{P}^2$, the distribution of Galois points in $\mathbb{P}^2 \setminus C$ was determined by the present author (see \cite{fukasawa1, fukasawa2}). 
This result can be rephrased as the following theorem on Galois lines for $\varphi(X)$. 

\begin{fact} \label{Galois point} 
Let $\ell \subset \mathbb{P}^3$ be a line with $\ell \ni (0:0:0:1)$ and $\ell \cap \varphi(X)=\emptyset$. 
Then, $\ell$ is a Galois line for $\varphi(X)$ if and only if $\ell \subset \{Z=0\}$ and $\ell$ is an $\mathbb{F}_q$-line.  
\end{fact}

\section{Proof of the if-part of Theorem \ref{main}} 
Let $\ell \subset \mathbb{P}^3$ be a line in (a), that is, let $\ell$ be an $\mathbb{F}_q$-line contained in the plane $\{Z=0\}$. 

We consider the case where $\ell$ passes through $(0:0:0:1)$. 
Since the defining polynomial is invariant under the action $(x, y) \mapsto (y, x)$, we can assume that $\ell$ is defined by $X-\alpha Y=Z=0$ for some $\alpha \in \mathbb{F}_q$. 
Assume that $\ell$ is defined by $X=Z=0$. 
Then, the projection $\pi_\ell$ from $\ell$ is represented by $(X:Y:Z:W) \mapsto (X:Z)$. 
The field extension given by $\pi_\ell$ is $k(x, y)/k(x)$ with an algebraic equation $y^q-y-c/(x^q-x)=0$. 
This extension is obviously Galois.  

Assume that $\ell$ is defined by $X-\alpha Y=Z=0$ with $\alpha^{q-1}=1$. 
The projection $\pi_\ell$ from $\ell$ is represented by $(X: Y:Z: W) \mapsto (X-\alpha Y:Z)$. 
Let $t=x-\alpha y$. 
Then, we have a field extension $k(y, t)/k(t)$ with an algebraic equation 
$$ \{\alpha (y^q-y)+(t^q-t)\}(y^q-y)-c=0. $$
We can assume that $\alpha=1$, since we have an extension $k(y, t')/k(t')=k(y, t)/k(t)$ for $t=\alpha t'$ with an equation 
$$ \{(y^q-y)+(t'^q-t')\}(y^q-y)-\frac{c}{\alpha}=0. $$
If $q$ is even, then the extension $k(y, t)/k(t)$ is Galois, since the algebraic equation is invariant under the actions 
$$ y \mapsto y+\beta, \ \mbox{ and } \ y \mapsto y+t, $$
where $\beta \in \mathbb{F}_q$. 
Assume that $q$ is odd. 
Let $u=y+t/2$. 
Then, $k(y, t)=k(u, t)$ and 
$$ (u^q-u)^2-\frac{(t^q-t)^2}{4}-c=0. $$
The extension $k(u, t)/k(t)$ is Galois, since this algebraic equation is invariant under the actions 
$$ u \mapsto u+\beta, \ \mbox{ and } \ u \mapsto -u, $$
where $\beta \in \mathbb{F}_q$.

We consider the case where $\ell$ does not contain $(0:0:0:1)$. 
Then, $\ell$ is defined by $Z=\alpha X+\beta Y+W=0$ for some $\alpha, \beta \in \mathbb{F}_q$. 
Since $\ell$ intersects $\varphi(\Omega_1)$ and $\varphi(\Omega_2)$, its follows that $\deg \pi_\ell \le 2q-2$. 
Then, the projection from $\ell$ is represented by $(1: \alpha x+\beta y+x y)$. 
Note that 
$$ \alpha x+\beta y+ x y=(x+\beta)(y+\alpha)-\alpha \beta. $$
Let $x'=x+\beta$ and $y'=y+\alpha$. 
Then, the extension $k(x, y)/k(\alpha x+\beta y+x y)$ coincides with $k(x', y')/k(x' y')$ with a relation 
$$ (x'^q-x')(y'^q-y)-c=0. $$
Then, the function $x' y'$ is invariant under the automorphisms 
$$ (x', y') \mapsto (\gamma x', \gamma^{-1} y'), \ \mbox{ and } \ (x', y') \mapsto (y', x') $$
of $C$, where $\gamma \in \mathbb{F}_q \setminus \{0\}$. 
Therefore, $[k(x', y'): k(x' y')] \ge 2(q-1)$. 
The equality holds, since $\deg \pi_{\ell} \le 2q-2$.  
Therefore, the extension $k(x', y')/k(x' y')$ is Galois. 

Assume that $\ell$ is defined by $W-a X=Y-a Z=0$ for some $a \in k$. 
Then, the projection $\pi_\ell$ is given by $(x y-a x: y-a)=(x: 1)$. 
The induced field extension is $k(x, y)/k(x)$, which coincides with the Galois extension of degree $q$ induced by the line $L_2: X=Z=0$. 
It is noted that $G_{\ell}=G_{L_2}$ in ${\rm Aut}(X)$. 

The proof of the if-part of Theorem \ref{main} is completed. 

\begin{rem} \label{infinitely many Galois lines} 
We can construct a model in $\mathbb{P}^3$ with infinitely many Galois lines, for each plane curve admitting a Galois point. 
Let $C \subset \mathbb{P}^2$ be a plane curve and let $P=(0:1:0)$ be a Galois point for $C$. 
Then, the extension $k(C)/\pi_P^*k(\mathbb{P}^1)=k(x, y)/k(x)$ is Galois. 
We consider a rational map 
$$ \varphi: C \dashrightarrow \mathbb{P}^2; \ (x: y: 1: x y), $$
which is birational onto its image. 
Then, the line $\ell$ defined by $W-a X=Y-a Z=0$ with $a \in k$ is a Galois line, since 
$$ \pi_\ell=(x y-a x: y-a)=(x: 1)$$
and $k(C)/\pi_\ell^*k(\mathbb{P}^1)=k(x, y)/k(x)$.  
Therefore, there exist infinitely many Galois lines for $\varphi(C)$.  
\end{rem}

Using the facts obtained in this section, we prove the following lemma, which is needed for the proof of the only-if-part of Theorem \ref{main}. 

\begin{lem} \label{lines} 
Let $H \subset \mathbb{P}^3$ be a hyperplane with $H \ne \{Z=0\}$. 
\begin{itemize} 
\item[(a)] If $H \supset \overline{PQ}$ and $H$ does not contain the tangent lines at $P$ and at $Q$, for some $P \in \varphi(\Omega_1)$ and $Q \in \varphi(\Omega_2)$, then points of the set $(\varphi(X) \cap H) \setminus \{P, Q\}$ are not collinear. 
\item[(b)] If $H \supset \varphi(\Omega_1)$, then $H$ contains the tangent line at some point of $\varphi(\Omega_1)$, or $q$ points of $(\varphi(X) \cap H) \setminus \varphi(\Omega_1)$ are collinear. 
For both cases, the defining equation of the tangent line or the line spanned by $q$ points is  of the from  
$$ W-a X=Y-a Z=0 $$
for some $a \in k$.  
\end{itemize}
\end{lem}

\begin{proof}
We consider (a). 
We can assume that $P=(1:0:0:0)$ and $Q=(0:1:0:0)$. 
Since the line $\overline{PQ}$ is a Galois line, and automorphisms  
$$ (x: y:1:x y) \mapsto (\gamma x: \gamma^{-1} y: 1: (\gamma x)(\gamma^{-1} y)) \ \mbox{ and } \ (x: y: 1: x y) \mapsto (y: x : 1: y x) $$
in the Galois group $G_{\overline{PQ}}$ are the restrictions of linear transformations
$$ (X: Y: Z: W) \mapsto (\gamma X: \gamma^{-1} Y: Z: W) \ \mbox{ and } (X:Y: Z: W) \mapsto (Y:X:Z:W), $$
it follows that $G_{\overline{PQ}}$ acts on each hyperplane containing $\overline{PQ}$. 
Assume by contradiction that the set $(\varphi(X) \cap H) \setminus \{P, Q\}$ is contained in some line $\ell \subset H$.  
Considering points on the plane $H$, it is inferred that $G_{\overline{PQ}}$ acts on $\ell$ also.  
However, there exists a nontrivial automorphism in $G_{\overline{PQ}}$ fixing $P$ and $Q$. 
Since such an automorphism does not fix $\ell$, this is a contradiction. 

We consider (b). 
Let $H$ be a hyperplane such that $H \supset \varphi(\Omega_1)$ and $H \ne \{Z=0\}$. 
Since $L_1$ is defined by $Y=Z=0$, $H$ is defined by $Y-a Z=0$ for some $a \in k$.  
If $a=\alpha \in \mathbb{F}_q$, then $H$ contains the tangent line at $\varphi(P_{\alpha})$. 
If $a \not\in \mathbb{F}_q$, then ${\rm supp}(\varphi^*H) \setminus \Omega_1$ consists of $q$ points defined by 
$$ (x^q-x)-\frac{c}{a^q-a}=0. $$  
The image of such $q$ points under $\varphi$ is contained in the line $W-a X=Y-a Z=0$.  
\end{proof}

\section{Proof of the only-if-part of Theorem \ref{main}} 
In this section, we assume that $\ell \subset \mathbb{P}^3$ is a Galois line for $\varphi(X)$. 
We prove in three steps. 

{\it (i) The case where $\ell \cap \varphi(X)=\emptyset$.}  
We prove that $\ell \ni (0:0:0:1)$. 
Assume by contradiction that $\ell \not\ni (0:0:0:1)$. 
Then, there exists a hyperplane $H \supset \ell$ such that $H \cap \{Z=0\} \not\ni (0:0:0:1)$.  
It follows from Fact \ref{automorphism} (a) and (b) that the Galois group $G_{\ell} \subset {\rm Aut}(X)$ acts on the plane $\{Z=0\}$. 
By Fact \ref{automorphism} (b) and the condition $\ell \cap \varphi(X)=\emptyset$, $G_\ell$ acts on the hyperplane $H$. 
These imply that $G_{\ell}$ acts on the line given by $H \cap \{Z=0\}$. 
Let $E:=\{(x, y) \mapsto (x+\alpha, y+\beta) \ | \ \alpha, \beta \in \mathbb{F}_q\} \subset {\rm Aut}(X)$. 
It follows from Fact \ref{automorphism} (c) that $|E \cap G_{\ell}|>1$, since $|E|=q^2$ and $|G_\ell|=\deg \pi_\ell =2q$. 
Therefore, there exists an automorphism $\sigma \in G_{\ell}$ of order $p$ such that $\sigma$ fixes the lines $L_1$ and $L_2$, where $L_i$ is the line spanned by $\varphi(\Omega_i)$ for $i=1, 2$. 
Since $H \cap \{Z=0\} \not\ni (0:0:0:1)$, it follows that any of non-collinear points given by $L_1 \cap L_2$, $L_2 \cap (H \cap \{Z=0\})$ and $(H \cap \{Z=0\}) \cap L_1$ is fixed by $\sigma$.  
Since there does not exist an element of $PGL(3, k)$ of order $p$ fixing non-collinear points, $\sigma$ acts on $\{Z=0\}$ trivially. 
This is a contradiction to Fact \ref{automorphism} (a). 
Therefore, $\ell \ni (0:0:0:1)$. 
It follows from Fact \ref{Galois point} that $\ell$ is an $\mathbb{F}_q$-line with $\ell \subset \{Z=0\}$. 
In other words, $\ell$ is of type (a) in Theorem \ref{main}. 

{\it (ii) The case where $\ell \cap \varphi(X)$ consists of a unique point and $\ell$ is not a tangent line.}  
In this case, the degree of the projection $\pi_\ell$ is $2q-1$, and $2q-1$ divides $|{\rm Aut}(X)|$. 
This is a contradiction to the fact that $|{\rm Aut}(X)|=2q^2(q-1)$ in Fact \ref{automorphism} (c). 

{\it (iii) The case where $\ell \cap \varphi(X)$ consists of two or more points, or $\ell$ is a tangent line.} 
If $\ell \subset \{Z=0\}$, then $\ell$ passes through two points of $\varphi(\Omega_1 \cup \Omega_2)$. 
In this case, $\ell$ is $\mathbb{F}_q$-rational, namely, $\ell$ is of type (a) in Theorem \ref{main}. 
Hereafter, we assume that $\ell \cap \{Z=0\}$ consists of a unique point $R$. 

Assume that $R \not\in L_1 \cup L_2$. 
If $R$ is not contained in any $\mathbb{F}_q$-line, then $\overline{RP} \cap \varphi(\Omega_1 \cup \Omega_2)=\{P\}$ for each $P \in \varphi(\Omega_1 \cup \Omega_2)$. 
It follows from Facts \ref{Galois extension} (a) and \ref{automorphism} (a) that $G_\ell$ fixes $P$ for each $P \in \varphi(\Omega_1 \cup \Omega_2)$. 
This is a contradiction to Fact \ref{automorphism} (a). 
Therefore, $R$ is contained in an $\mathbb{F}_q$-line. 
Then, there exist points $P \in \varphi(\Omega_1)$ and $Q \in \varphi(\Omega_2)$ such that $\overline{RP}=\overline{RQ}$. 
Let $H$ be a plane spanned by $\ell$ and $Q$. 
If $H$ contains the tangent lines at $P$ and at $Q$, then by Fact \ref{orders}, $\ell \cap \varphi(X)$ is an empty set. 
Therefore, using Fact \ref{Galois extension} (2), it follows that $H$ is not a tangent hyperplane at $P$ or at $Q$. 
Since the Galois group $G_{\ell}$ acts on $\{P, Q\}$ by Facts \ref{Galois extension} (a) and \ref{automorphism} (a), the set $(\varphi(X) \cap H) \setminus \{P, Q\}$ is contained in the line $\ell$. 
This is a contradiction to Lemma \ref{lines} (a).

Assume that $R$ is contained in the line $L_1$. 
The plane spanned by $L_1$ and $\ell$ contains $\varphi(\Omega_1)$. 
Since $\ell$ contains two points or is a tangent line, it follows from Lemma \ref{lines} (b) that $\ell$ is defined by $W-a X=Y- a Z=0$ for some $a \in k$, that is, $\ell$ is of type (b) in Theorem \ref{main}.

\

\begin{center}
{\bf Acknowledgments} 
\end{center}
The author is grateful to Professor G\'{a}bor Korchm\'{a}ros for a helpful discussion in the International Conference ``Combinatorics 2018'' held at Arco, Italy. 
The author thanks Doctor Kazuki Higashine for helpful comments for this work.

\end{document}